\documentclass[10pt,reqno]{amsart}
\usepackage{amsmath}
\usepackage{amssymb}
\usepackage{amsthm}
\usepackage{eepic,epic}
\usepackage{epsfig}
\usepackage{graphicx}

\newlength{\defbaselineskip}
\newcommand{\setlinespacing}[1]%
           {\setlength{\baselineskip}{#1 \defbaselineskip}}

\numberwithin{equation}{section}

\newtheorem{thm}{Theorem}[section]

\newtheorem{lem}[thm]{Lemma}
\newtheorem{prop}[thm]{Proposition}

\theoremstyle{definition}

\theoremstyle{remark}

\numberwithin{equation}{section}

\begin{document}

\title[The energy-critical INLS]
{The Cauchy problem for the energy-critical inhomogeneous nonlinear Schr\"{o}dinger equation}

\author{Yoonjung Lee and Ihyeok Seo}

\thanks{This research was supported by NRF-2019R1F1A1061316.}


\address{Department of Mathematics, Sungkyunkwan University, Suwon 16419, Republic of Korea}
\email{yjglee@skku.edu}
\email{ihseo@skku.edu}

\begin{abstract}
In this paper we study the Cauchy problem for the energy-critical inhomogeneous nonlinear Schr\"odinger equation
$i\partial_{t}u+\Delta u=\lambda|x|^{-\alpha}|u|^{\beta}u$ in $H^1$.
The well-posedness theory in $H^1$ has been intensively studied in recent years,
but the currently known approaches do not work for the critical case $\beta=(4-2\alpha)/(n-2)$.
It is still an open problem.
The main contribution of this paper is to develop the theory in this case.
\end{abstract}

\maketitle

\section{Introduction}
In this paper we consider the Cauchy problem for the inhomogeneous nonlinear Schr\"odinger equation (INLS)
\begin{equation}\label{INLS}
\left\{
\begin{aligned}
&i \partial_{t} u + \Delta u = \lambda |x|^{-\alpha} |u|^{\beta} u,\quad (x,t) \in \mathbb{R}^n \times \mathbb{R}, \\
&u(x, 0)=u_0 \in H^1,
\end{aligned}
\right.
\end{equation}
where $0<\alpha<2$, $\beta>0$ and $\lambda=\pm1$.
Here, the case $\lambda = 1$ is \textit{defocusing}, while the case $\lambda=-1$ is \textit{focusing}.
This model arises naturally in various physical contexts such as nonlinear optics and plasma physics
for the propagation of laser beams in an inhomogeneous medium (\cite{B,TM}).
This equation enjoys the scale-invariance $u(x,t)\mapsto u_{\delta}(x,t)= \delta^{\frac{2-\alpha}{\beta}} u(\delta x, \delta^2 t)$
for $\delta>0$, and
$\|u_{{\delta}, 0}\|_{\dot{H}^1} = \delta^{1+\frac{2-\alpha}{\beta}-\frac{n}{2}}\|u_0 \|_{\dot{H}^1}$
where $u_{{\delta}, 0}$ denotes rescaled initial data.
If $\beta=(4-2\alpha)/(n-2)$, the scaling preserves the $\dot{H}^1$ norm of $u_0$
and in this case \eqref{INLS} is called the energy-critical INLS.

The case $\alpha = 0$ in \eqref{INLS} is the classical nonlinear Schr\"odinger equation (NLS)
whose well-posedness theory in the energy space $H^1$ has been extensively studied over the past several decades
and is well understood
(see, for example, \cite{K,GV3,GV4} for the subcritical case, $\beta<4/(n-2)$, and \cite{CW} for the critical case, $\beta=4/(n-2)$).
However, much less is known about the INLS which has drawn attention in recent years.
In particular, the critical case $\beta=(4-2\alpha)/(n-2)$ is still an open problem.
The main contribution of this paper is to develop the theory in this case.

Let us first review some known results for the subcritical case, $\beta<(4-2\alpha)/(n-2)$.
We shall assume $n\geq3$ to make the review shorter.
Genoud and Stuart \cite{GS} first studied \eqref{INLS} for the focusing case in the sense of distribution.
Using the abstract argument of Cazenave \cite{C} which does not use Strichartz estimates,
they showed that \eqref{INLS} is well-posed locally, and globally for small initial data,
with the full range of $0<\alpha<2$.
In this case, Farah \cite{F} also showed how small should be the initial data to have global well-posedness
in the spirit of Holmer-Roudenko \cite{HR} for the NLS.
Recently, Guzm\'{a}n \cite{GU} used the contraction mapping argument relying on the classical known Strichartz estimates,
from which the above-mentioned results are obtained with a restriction ($0<\alpha<1$) on the validity of $\alpha$ when $n=3$.
This restriction is a bit improved by Dinh \cite{D} to $1\leq\alpha<3/2$ but for more restricted values $\beta<(6-4\alpha)/(2\alpha-1)$.
Although these results are a bit weak on the validity of $\alpha$ when $n=3$ compared with the result of Genoud-Stuart,
but they provide more information on the solution due to the Strichartz estimates.
In particular, one can know that the solution belongs to $L_t^{q}H_x^{1,r}$ for \textit{Schr\"odinger-admissible} pairs $(q,r)$
for which the Strichartz estimates hold.
In general, such property plays an important role in studying other interesting problems, for instance, scattering and blow up.

However, the critical case $\beta=(4-2\alpha)/(n-2)$ is still an open problem.
In this paper we develop the well-posedness theory in this case.
To this end, we approach to the matter by thinking of the following weighted space-time norms with $\gamma\geq0$
\begin{equation*}
\|f\|_{L_{t}^q L_{x}^r(|x|^{-r\gamma})}=  \bigg( \int_{\mathbb{R}} \bigg(\int_{\mathbb{R}^n}  |x|^{-r\gamma} |f(x)|^r  dx \bigg)^{\frac{q}{r}} dt\bigg)^{\frac{1}{q}}.
\end{equation*}
This weighted space approach was first developed in our previous work \cite{KLS} in which the $L^2$-critical INLS was first solved.
But here we proceed slightly differently from that to make the approach worked for the energy-critical case as well.
Our approach does seem to be more suitable to perform a finer analysis for the INLS model
because the singularity $|x|^{-\alpha}$ in the nonlinear term can be handled more effectively in the weighted setting.

Before stating our results, we first introduce the following weighted Sobolev norms\footnote{We may replace the norm with
$\|f\|_{H^{1,r}(|x|^{-r\gamma})} = \| (1+|\nabla|^2)^{1/2} f \|_{L^r (|x|^{-r\gamma})}$.
Indeed, the two norms coincide with each other if $1<r<\infty$ and $0<r\gamma<n$.
This can be shown by a standard process using a weighted version (Lemma 12.1.4, \cite{MS}) of Mikhlin's multiplier theorem.}
$$\|f\|_{H^{1,r}(|x|^{-r\gamma})} = \|f\|_{L^r(|x|^{-r\gamma})}+\|\nabla f\|_{L^r(|x|^{-r\gamma})}.$$
Our first result is then the following theorem.

\begin{thm}\label{Thm1}
Let $n \geq 3$ and $\beta=(4-2\alpha)/(n-2)$ for $0 < \alpha < \min \lbrace 2, n/2 \rbrace $.
Assume that $\gamma$ satisfy
\begin{align*}
 \frac{\alpha-\beta-1}{\beta+1}  < \gamma < \min \bigg\{ \frac{n-2}{2}, \frac{\alpha}{\beta+1} \bigg\}.
\end{align*}
If $u_0 \in H^1$,
 there exist\, $T>0$ and a unique local solution of the problem \eqref{INLS} with
\begin{equation*}
u \in C([0,T] ; H^{1}) \cap L^{q}([0,T] ; H^{1,r}(|x|^{-r \gamma}))
\end{equation*}
for any $\gamma$-Schr\"odinger admissible pairs $(q, r)$;
\begin{equation}\label{ad}
0<\gamma<1,\quad \frac{\gamma}{2} < \frac{1}{q}\leq \frac12  \quad\text{and}\quad  \frac{2}{q}=n(\frac{1}{2}-\frac{1}{r})+\gamma.
\end{equation}
\end{thm}

We also provide the small data global well-posedness and the scattering results for the energy-critical INLS.
In the critical case, the local solution exists in a time interval depending on the data $u_0$ itself and not on its norm.
Therefore, the conservation laws do not guarantee the existence of a global solution any more.
For this reason, $\|u_0\|_{H^1}$ is usually assumed to be small.

\begin{thm}\label{thm2}
Under the same conditions as in Theorem \ref{Thm1} and the smallness assumption on $\|u_0\|_{H^1}$,
there exists a unique global solution of the problem \eqref{INLS} with
$$u \in C([0,\infty) ; H^{1}) \cap L^{q}([0,\infty) ; H^{1,r}(|x|^{-r \gamma}))$$
for any $\gamma$-Schr\"odinger admissible pairs $(q, r)$.
Furthermore, the solution scatters in $H^1$,
i.e., there exist $\phi \in H^1$ such that
$$\lim_{t \rightarrow \infty} \| u(t)-e^{it\Delta} \phi \|_{H_x^1} =0.$$
\end{thm}

The argument in this paper can be also applied to the subcritical case $\beta< (4-2\alpha)/(n-2)$
with the same validity of $\alpha$,
and therefore this improves the above-mentioned results of Guzm\'{a}n \cite{GU} and
Dinh \cite{D} on the validity of $\alpha$.
But this is not a purpose of the present paper.
Particularly when $\alpha=0$ in our approach, it is deduced that $\gamma=0$,
and in this case resulting results also cover the classical results (\cite{GV3, GV4, K, CW}) for the NLS equation.

To prove the theorems, we first obtain some weighted estimates for the nonlinear term
relying on weighted Strichartz estimates. See Section \ref{sec2}.
These estimates will play a crucial role in Section \ref{sec3},
when proving the well-posedness results by applying the contraction mapping argument along with the Strichartz estimates.

Throughout this paper, the letter $C$ stands for a positive constant which may be different
at each occurrence.
We also denote $A\lesssim B$ to mean $A\leq CB$
with unspecified constants $C>0$.

\section{Weighted estimates}\label{sec2}
This section contains some weighted estimates needed for the proofs of Theorems \ref{Thm1} and \ref{thm2} in the next section.

\subsection{Strichartz estimates}
One of the most basic tools for the well-posedness of nonlinear dispersive equations is
the contraction mapping principle.
The key ingredient in this argument is the availability of Strichartz estimates.
In our case we need to obtain the estimates in the weighted setting.

Before stating them, we introduce some notations.
For $0<\gamma<1$, we set
$$A_{\gamma}=\lbrace (q, r) : (q, r) \text{ is } \gamma\text{-Schr\"odinger admissible} \rbrace,$$
and then define the weighted Stichartz norm
$$\|u\|_{S_\gamma(I)}:= \sup_{(q,r)\in A_\gamma} \big\| |x|^{-\gamma} u\big\|_{L_t^q(I; L_x^r)}$$
and its dual weighted Strichartz norm
$$\|v\|_{S_{\tilde{\gamma}}'(I)}:= \inf_{(\tilde{q},\tilde{r})\in A_{\tilde{\gamma}}, \tilde{q}>2}
\big\| |x|^{ \tilde{\gamma}} v\big\|_{L_t^{\tilde{q}'}(I; L_x^{\tilde{r}'})}$$
for any interval $I\subset\mathbb{R}$.
Now we state the weighted Strichartz estimates:

\begin{prop}\label{pr}
Let $n\ge 3$.
Then we have
\begin{equation}\label{p1}
\|e^{it\Delta} f\|_{S_\gamma(I)} \lesssim \|f\|_{L_x^2},
\end{equation}
\begin{equation}\label{p11}
\bigg\| \int_{-\infty}^\infty e^{-i\tau\Delta}F(\tau) d\tau \bigg\|_{L_x^2} \lesssim  \|F\|_{S_{\tilde{\gamma}}'(I)}
\end{equation}
and
\begin{equation}\label{p2}
\bigg\| \int_{0}^{t} e^{i(t-\tau)\Delta}F(\tau) d\tau \bigg\|_{S_\gamma(I)} \lesssim  \|F\|_{S_{\tilde{\gamma}}'(I)}.
\end{equation}
\end{prop}

\begin{proof}
Let $n\ge 3$ and $0< \gamma, \tilde{\gamma}<1$.
To show the first estimate \eqref{p1}, we may show
\begin{equation}\label{ws}
		\left\| e^{it\Delta} f \right\|_{L_t^{q}(I; L_x^{r}(|x|^{-r\gamma}))} \lesssim \left\| f \right\|_{L^{2}}
\end{equation}
for any $(q, r)\in A_{\gamma}$.
For this we first recall the classical Strichartz estimates
\begin{equation}\label{st}
	\left\| e^{it\Delta} f \right\|_{L_t^{a}(I; L_x^{b})} \lesssim \left\| f \right\|_{L^{2}}
\end{equation}
which holds if and only if $2/a=n(1/2-1/b)$ and $2 \leq a \leq \infty$.
It was first established by Strichartz \cite{St} for the diagonal case $q=r$ and then extended to mixed norms completely as in \eqref{ws}
(\cite{GV4, KT}).
We also need to make use of the Kato-Yajima smoothing estimates
\begin{equation}\label{kt}
	\big\| |\nabla|^s e^{it\Delta} f \big\|_{L_t^2(I;L_x^2(|x|^{-2(1-s)}))} \lesssim \| f \|_{L^{2}}
\end{equation}
which holds if and only if $-(n-2)/2 < s <1/2$.
Kato and Yajima \cite{KY} first discovered this estimate for $0\leq s <1/2$ (see also \cite{BK} for an alternate proof)
and the optimal range was obtained later (\cite{W,Su,V}).

We now deduce \eqref{ws} from using the complex interpolation between \eqref{st} and \eqref{kt}
by appealing to the following complex interpolation space identities.

\begin{lem}[\cite{BL}]\label{idd}
Let $0<\theta<1$ and $1\leq p_0,p_1<\infty$. Given two complex Banach spaces $A_0$ and $A_1$,
$$(L^{p_0}(A_0),L^{p_1}(A_1))_{[\theta]}=L^p((A_0,A_1)_{[\theta]})$$
if $1/p=(1-\theta)/p_0+\theta/p_1$, and if $w=w_0^{p(1-\theta)/p_0}w_1^{p\theta/p_1}$
$$(L^{p_0}(w_0),L^{p_1}(w_1))_{[\theta]}=L^p(w).$$
Here, $(\cdot\,,\cdot)_{[\theta]}$ denotes the complex interpolation functor.
\end{lem}

In fact, using the complex interpolation between \eqref{st} and \eqref{kt} with $s=0$, we first see
\begin{equation*}
	\left\| e^{it\Delta} f \right\|_{\big(L_t^{a}(I; L_x^{b}) ,L_t^2(I;L_x^2(|x|^{-2}))\big)_{[\theta]}} \lesssim \left\| f \right\|_{L^{2}},
\end{equation*}
and by Lemma \ref{idd} we then obtain
\begin{equation*}
	\left\| e^{it\Delta} f \right\|_{L_t^{q}(I; L_x^{r}(|x|^{-r\gamma}))} \lesssim \left\| f \right\|_{L^{2}}
\end{equation*}
where
\begin{equation}\label{con}
\frac{1}{q}=\frac{1-\theta}{a}+\frac{\theta}{2},\quad \frac{1}{r}=\frac{1-\theta}{b}+\frac{\theta}{2}\quad\text{and}\quad
\gamma=\theta
\end{equation}
under the conditions
\begin{equation}\label{con2}
\frac{2}{a}=n(\frac{1}{2}-\frac{1}{b}),\quad 0 < \frac{1}{a}\leq \frac{1}{2} \quad \textnormal{and}\quad 0<\theta<1.
\end{equation}
By replacing $\theta$ with $\gamma$, and then eliminating the redundant exponents $a,b$,
it is not difficult to see that the requirements \eqref{con} and \eqref{con2} are reduced to \eqref{ad}.
Hence we obtain the desired estimate \eqref{ws}.
The second estimate \eqref{p11} follows now from the adjoint form of \eqref{ws}:
\begin{equation}\label{adj}
\left\| \int_{-\infty}^\infty  e^{-i\tau\Delta} F(\tau) d\tau \right\|_{L_x^2} \lesssim
 \| F \|_{L_t^{\tilde{q}'}(I; L_x^{\tilde{r}'}(|x|^{\tilde{r}'\tilde{\gamma}}))}
\end{equation}
for any $(\tilde{q}, \tilde{r})\in A_{\tilde{\gamma}}$.

It remains to show \eqref{p2}.
By the standard $TT^\ast$ argument, \eqref{ws} and \eqref{adj} imply
\begin{equation}\label{wi}
		\left\| \int_{-\infty}^\infty  e^{i(t-\tau)\Delta} F(\cdot, \tau) d\tau \right\|_{L_t^{q}(I; L_x^{r}(|x|^{-r\gamma}))}
\lesssim \left\| F \right\|_{L_t^{\tilde{q}'}(I; L_x^{\tilde{r}'}(|x|^{\tilde{r}'\tilde{\gamma}}))}
\end{equation}
for any $(q, r)\in A_{\gamma}$ and $(\tilde{q}, \tilde{r})\in A_{\tilde{\gamma}}$.
If we further assume $\tilde{q}>2$ (and hence $q>\tilde{q}'$),
we may apply the Christ-Kiselev lemma \cite{CK} to get \eqref{wi} with $\int_{-\infty}^\infty$ replaced by $\int_{0}^t$.
This completes the proof.
\end{proof}

\subsection{Estimates for the nonlinear term}
In this subsection we establish some useful weighted estimates for the nonlinearity $|x|^{-\alpha} |u|^{\beta}u$
in the weighted Strichartz spaces by making use of Proposition \ref{pr} and a special case of
Caffarelli-Kohn-Nirenberg weighted interpolation inequalities.

\begin{lem}\label{le}
Let $n\ge3$ and $\beta=(4-2\alpha)/(n-2)$ for $0< \alpha < \min \lbrace 2, n/2 \rbrace$.
Assume that
\begin{equation}\label{gg}
 \frac{\alpha -\tilde{\gamma}-\beta}{\beta+1} \leq \gamma
 \leq \frac{\alpha -\tilde{\gamma}}{\beta+1}\quad \textnormal{and} \quad \gamma<\frac{n-2}{2}.
\end{equation}
Then we have
\begin{equation*}
\big\||x|^{-\alpha} |u|^{\beta}v\big\|_{S_{\tilde{\gamma}}'(I)} \leq C\| \nabla u\|_{S_{\gamma}(I)}^{\beta}
\| v\|_{S_{\gamma}(I)}
\end{equation*}
and
\begin{equation*}
\|\nabla (|x|^{-\alpha} |u|^{\beta}u)\|_{S_{\tilde{\gamma}}'(I)}
\leq C\| \nabla u\|_{S_{\gamma}(I)}^{\beta+1}.
\end{equation*}
\end{lem}

\begin{proof}
It is sufficient to show that there exist $(q, r) \in A_{\gamma}$ and
$(\tilde{q}, \tilde{r}) \in A_{\tilde{\gamma}}$ with $\tilde{q}>2$
for which
\begin{equation}\label{31}
\big\||x|^{-\alpha} |u|^{\beta}v\big\|_{L_t^{\tilde{q}'}(I;L_x^{\tilde{r}'} (|x|^{\tilde{r}' \tilde{\gamma}}))}
\leq C\||x|^{-\gamma} \nabla u\|_{L_t^{q}(I;L_x^{r})}^{\beta}\| |x|^{-\gamma}v\|_{L_t^{q}(I;L_x^{r})}
\end{equation}
and
\begin{equation}\label{42}
\|\nabla (|x|^{-\alpha} |u|^{\beta}u)\|_{L_t^{\tilde{q}'}(I;L_x^{\tilde{r}'} (|x|^{\tilde{r}' \tilde{\gamma}}))}
\leq C\||x|^{-\gamma} \nabla u\|_{L_t^{q}(I;L_x^{r})}^{\beta+1}
\end{equation}
hold for $\alpha,\beta,\gamma,\tilde{\gamma}$ given as in the lemma.
Let us first set
\begin{equation}\label{se}
\frac{1}{\tilde{q}'}=\frac{\beta+1}{q} \quad \textnormal{and} \quad \frac{1}{\tilde{r}'}=\frac1r+\frac1{r_1}
\end{equation}
with
\begin{equation}\label{lll}
\frac{1}{r_1}=\frac{\beta}{r}-\frac{\gamma (\beta+1)+\tilde{\gamma}-\alpha+\beta}{n},
\end{equation}
by which we easily see that $\beta=(4-2\alpha)/(n-2)$ if $(q, r) \in A_{\gamma}$ and $(\tilde{q}, \tilde{r})\in A_{\tilde{\gamma}}$.

\subsubsection*{Proof of \eqref{31}}

Using H\"older's inequality with \eqref{se}, we first see
\begin{align}\label{non0}
\nonumber\big\||x|^{-\alpha} |u|^{\beta}v\big\|_{L_t^{\tilde{q}'}(I;L_x^{\tilde{r}'} (|x|^{\tilde{r}' \tilde{\gamma}}))}
&=  \big\| |x|^{\tilde{\gamma}-\alpha} |u|^{\beta} v \big\|_{L_t^{\tilde{q}'}(I;L_x^{\tilde{r}'})} \\
&\leq  \big\| |x| ^{\gamma+\tilde{\gamma}-\alpha} |u|^{\beta}\big\|_{L_t^{\frac{q}{\beta}}(I;L_x^{r_1})}
\big\| |x| ^{-\gamma} v \big\|_{L_t^{q}(I;L_x^{r})}.
\end{align}
We shall then use the following special case of
Caffarelli-Kohn-Nirenberg weighted interpolation inequalities.

\begin{lem}[\cite{CKN}] \label{le3}
Let $n\ge 1$. If
\begin{equation*}
1\leq p\leq q<\infty,\quad -n/q<b\leq a \quad\text{and}\quad a-b-1=n/q-n/p,
\end{equation*}
then
\begin{equation}\label{em}
\big\||x|^{b}f\big\|_{L^q} \leq C \big\||x|^{a} \nabla f\big\|_{L^p}.
\end{equation}
\end{lem}

Applying \eqref{em} to the first term on the right-hand side of \eqref{non0} with
$b=\frac{\gamma+\tilde{\gamma}-\alpha}{\beta}$, $q=\beta r_1$, $a=-\gamma$ and $p=r$,
we get
\begin{equation}\label{qwd}
\big\| |x| ^{\gamma+\tilde{\gamma}-\alpha} |u|^{\beta}\big\|_{L_t^{\frac{q}{\beta}}(I;L_x^{r_1})}
= \big\| |x| ^{\frac{\gamma+\tilde{\gamma}-\alpha}{\beta}} u\big\|_{L_t^q(I;L_x^{\beta r_1})}^{\beta}
\leq C\big\| |x| ^{-\gamma} \nabla u\big\|_{L_t^q(I;L_x^r)}^{\beta}
\end{equation}
if
\begin{equation*}
 0< \frac{1}{\beta r_1} \leq \frac{1}{r} \leq 1,\quad
 -\frac{n}{\beta r_1}  < \frac{\gamma+\tilde{\gamma}-\alpha}{\beta} \leq -\gamma,\quad
-\gamma- \frac{\gamma+\tilde{\gamma}-\alpha}{\beta} -1=\frac{n}{\beta r_1}-\frac{n}{r}.
\end{equation*}
Using \eqref{lll}, this requirement is reduced to
\begin{equation}\label{rr2}
0< \frac{1}{r}-\frac{\gamma (\beta+1)}{n\beta}-\frac{\tilde{\gamma}-\alpha+\beta}{n\beta} \leq \frac{1}{r} \leq 1
\end{equation}
and
\begin{equation}\label{rr3}
-\frac{n}{r}+\gamma +\frac{\gamma+\tilde{\gamma}-\alpha}{\beta}+1  < \frac{\gamma+\tilde{\gamma}-\alpha}{\beta} \leq -\gamma.
\end{equation}
We note here that $1/r \leq1$ is trivially satisfied if $(q, r)\in A_\gamma$.
The first inequality of \eqref{rr2}, which is $1+\gamma+\frac{\tilde{\gamma}+\gamma-\alpha}{\beta} <\frac{n}{r}$,
can be also removed by the first two inequalities of \eqref{rr3}.
Consequently, \eqref{rr2} and \eqref{rr3} are reduced to
\begin{equation}\label{c0}
\gamma \ge \frac{\alpha -\tilde{\gamma} -\beta}{\beta+1},\quad
1+\gamma < \frac{n}{r} \quad\text{and}\quad \gamma \leq \frac{\alpha -\tilde{\gamma}}{\beta+1}
\end{equation}
where the first and third ones give the first restriction on $\gamma$ in \eqref{gg}.
Meanwhile, if $(q, r) \in A_\gamma$ we see
\begin{equation}\label{c1}
 \frac{1}{2}-\frac{1-\gamma}{n} \leq \frac{1}{r} < \frac{1}{2}
\end{equation}
by combining the conditions in \eqref{ad}.
Similarly, if $(\tilde{q}, \tilde{r}) \in A_{\tilde{\gamma}}$ with $\tilde{q}>2$, we see
$1/2-(1-\tilde{\gamma})/n< 1/\tilde{r} < 1/2$
which implies
\begin{equation}\label{c2}
\frac{1}{2(\beta+1)}+\frac{\gamma}{n}+\frac{\tilde{\gamma}+\beta-\alpha}{n(\beta+1)} < \frac{1}{r} < \frac{1}{2(\beta+1)}+\frac{\gamma}{n}+\frac{1+\beta-\alpha}{n(\beta+1)}
\end{equation}
by using the second condition in \eqref{se} with \eqref{lll}.

Now we have to show that there exists $r$  satisfying the second condition of \eqref{c0}, \eqref{c1} and \eqref{c2} simultaneously
under the first condition in \eqref{gg}.
For this, we make all the lower bounds of $1/r$ be less than each upper one.
We start with the upper bound of \eqref{c1} to compare the lower ones of $1/r$ in \eqref{c0}, \eqref{c1} and \eqref{c2} in turn,
and then the conditions $ \gamma < (n-2)/2$, $\gamma<1$ and $\gamma < \frac{(n-2)\beta}{2(\beta+1)}+ \frac{\alpha-\tilde{\gamma}}{\beta+1}$ follow respectively.
But all these requirements are trivially valid by the assumption \eqref{gg}.
Similarly, using the upper bound of \eqref{c2} implies
$\alpha<n/2$, $(n-4)\beta < 4-2\alpha $ and $\tilde{\gamma}<1$.
Here the last two conditions are trivially satisfied and hence $\alpha<n/2$ is only required.
Therefore, by combining \eqref{non0} and \eqref{qwd}, we obtain \eqref{31} as desired.

\subsubsection*{Proof of \eqref{42}}
We first see that
\begin{align}\label{qaz}
\nonumber\|\nabla (&|x|^{-\alpha} |u|^{\beta}u)\|_{L_t^{\tilde{q}'}(I;L_x^{\tilde{r}'} (|x|^{\tilde{r}' \tilde{\gamma}}))}\\
\nonumber&\lesssim \big\||x|^{-\alpha}|u|^{\beta} \nabla u\big\|_{L_t^{\tilde{q}'}(I;L_x^{\tilde{r}'} (|x|^{\tilde{r}' \tilde{\gamma}}))}
+\big\| |x|^{-\alpha-1} |u|^{\beta}u\big\|_{L_t^{\tilde{q}'}(I;L_x^{\tilde{r}'} (|x|^{\tilde{r}' \tilde{\gamma}}))}\\
&:= B_1 +B_2.
\end{align}
The first term $B_1$ is bounded by using H\"older's inequality with \eqref{se} as
\begin{align*}
B_1&= \big\| |x|^{\tilde{\gamma}-\alpha} |u|^{\beta}\nabla u\big\|_{L_t^{\tilde{q}'}(I; L_x^{\tilde{r}'})}\\
&\leq \big\||x|^{\tilde{\gamma}+\gamma-\alpha} |u|^{\beta}\big \|_{L_t^{\frac{q}{\beta}} (I;L_x^{r_1})}
\big\| |x|^{-\gamma} \nabla  u\big\|_{L_t^q(I; L_x^r)}.
\end{align*}
By applying Lemma \ref{le3} as above (see \eqref{qwd}), we get
\begin{equation}\label{b11}
B_1 \leq C \big\| |x|^{-\gamma}\nabla u\big\|_{L_t^q(I; L_x^r)}^{\beta+1}
\end{equation}
under the same conditions as in the proof of \eqref{31}.
Similarly,
\begin{align*}
\nonumber B_2&=\big\| |x|^{\tilde{\gamma}-\alpha-1} |u|^{\beta} u\big\|_{L_t^{\tilde{q}'}(I; L_x^{\tilde{r}'})}\\
&\leq \big\| |x| ^{\gamma+\tilde{\gamma}-\alpha} |u|^{\beta}\big\|_{L_t^{\frac{q}{\beta}}(I;L_x^{r_1})}
\big\| |x| ^{-\gamma-1} u \big\|_{L_t^{q}(I;L_x^{r})}\\
&\leq C\big\| |x|^{-\gamma} \nabla u\big\|_{L_t^{q}(I; L_x^{r})}^{\beta}\big\| |x| ^{-\gamma-1} u\big\|_{L_t^{q}(I;L_x^{r})}.
\end{align*}
Here we apply Lemma \ref{le3} with $b=-\gamma-1$, $q=r$, $a=-\gamma$, $p=r$ to obtain
\begin{equation*}
\big\| |x| ^{-\gamma-1} u \big\|_{L_t^{q}(I;L_x^{r})}\leq C\big\| |x| ^{-\gamma}\nabla u \big\|_{L_t^{q}(I;L_x^{r})}
\end{equation*}
where $-n/r<-\gamma-1\leq-\gamma$ is required.
But the first inequality is just the second one in \eqref{c0} and the last inequality is trivially satisfied. Hence,
\begin{equation}\label{b22}
 B_2\leq C\big\| |x|^{-\gamma} \nabla u\big\|_{L_t^{q}(I; L_x^{r})}^{\beta+1}
\end{equation}
under the same conditions as in the proof of \eqref{31}.
Consequently, we obtain the desired estimate \eqref{42}
combining \eqref{qaz}, \eqref{b11} and \eqref{b22}.
\end{proof}

\section{The well-posedness in $H^1$}\label{sec3}
In this section we prove Theorems \ref{Thm1} and \ref{thm2} by applying the contraction mapping principle.
The weighted estimates in the previous section play a key role in this step.

\subsubsection*{Proof of Theorem \ref{Thm1}}
By Duhamel's principle, we first write the solution of the Cauchy problem \eqref{INLS} as
\begin{equation*}
\Phi_{u_0}(u) = e^{it \Delta} u_0 - i\lambda  \int_{0}^{t} e^{i(t-\tau)\Delta}F(u) d\tau
\end{equation*}
where $F(u)=|\cdot|^{-\alpha} |u(\cdot,\tau)|^{\beta}u(\cdot, \tau)$.
For appropriate values of $T, M, N>0$, we shall show that $\Phi$ defines a contraction map on
\begin{align*}
X(T,M,N) =\big\lbrace u\in C_t(I;H_x^1)\cap &L_t^q(I;H_x^{1,r}(|x|^{-r\gamma})):\\
&\sup_{t\in I}\|u\|_{H_x^1}\leq N,\, \|u\|_{\mathcal{H}_{\gamma}(I)} \leq M \big\rbrace
\end{align*}
equipped with the distance
$$d(u,v)=\sup_{t\in I}\|u-v\|_{L_x^2}+\|u-v\|_{S_\gamma(I)}.$$
Here, $I=[0, T]$, and $(q, r,\gamma)$ is given as in the theorems.
We also define
$$\| u\|_{\mathcal{H}_\gamma (I)} := \|u\|_{S_\gamma(I)}+\|\nabla u\|_{S_\gamma(I)}$$
and
$$\| u\|_{\mathcal{H}_{\tilde{\gamma}}' (I)} := \|u\|_{S_{\tilde{\gamma}}'(I)}+\|\nabla u\|_{S_{\tilde{\gamma}}'(I)}.$$
Then $(X,d)$ is a complete metric space, which will be shown in the next section.

We now show that $\Phi$ is well defined on $X$.
By Proposition \ref{pr}, we get
\begin{equation}\label{zx}
\|\Phi(u)\|_{\mathcal{H}_\gamma(I)} \leq \| e^{it\Delta} u_0\|_{\mathcal{H}_\gamma(I)}+C\|F(u)\|_{\mathcal{H}'_{\tilde{\gamma}}(I)}
\end{equation}
and
\begin{equation*}
\sup_{t\in I} \|\Phi(u)\|_{H_x^1} \leq C\|u_0\|_{H^1}+
\sup_{t\in I} \bigg\|\int_{0}^{t} e^{i(t-\tau)\Delta}F(u) d\tau\bigg\|_{H_x^1}.
\end{equation*}
Since $\|f\|_{H^1}\lesssim\|f\|_{L^2}+\|f\|_{\dot{H}^1}$, using the fact that $e^{it\Delta}$ is an isometry on $L^2$ and $\dot{H}^1$,
and then applying \eqref{p11}, we see that
\begin{align*}
\sup_{t\in I} \bigg\|\int_{0}^{t} e^{i(t-\tau)\Delta}F(u) d\tau\bigg\|_{H_x^1}
\lesssim\|F(u)\|_{S_{\tilde{\gamma}}'(I)}+\|\nabla F(u)\|_{S_{\tilde{\gamma}}'(I)}.
\end{align*}
Hence,
\begin{equation}\label{zx2}
\sup_{t\in I} \|\Phi(u)\|_{H_x^1} \leq C\|u_0\|_{H^1}+C\|F(u)\|_{\mathcal{H}_{\tilde{\gamma}}'(I)}.
\end{equation}
On the other hand, using Lemma \ref{le}, we get
\begin{align}\label{asb}
\nonumber\|F(u) \|_{\mathcal{H}'_{\tilde{\gamma}}(I)} &\leq C \| \nabla u\|_{S_{\gamma}(I)}^{\beta} \|u\|_{S_{\gamma}(I)}+C\|
\nabla u\|_{S_{\gamma}(I)}^{\beta+1}\\
\nonumber&\leq C \| \nabla u\|_{S_{\gamma}(I)}^{\beta} \|u\|_{\mathcal{H}_{\gamma}(I)}\\
&\leq CM^{\beta+1}
\end{align}
if $u\in X$, and for some $\varepsilon>0$ small enough which will be chosen later we get
\begin{equation}\label{ep}
\|e^{it\Delta}u_0\|_{\mathcal{H}_{\gamma}(I)}\leq \varepsilon
\end{equation}
which holds for a sufficiently small $T>0$ by the dominated convergence theorem.
We now conclude that
$$\|\Phi(u)\|_{\mathcal{H}_\gamma(I)} \leq \varepsilon + CM^{\beta+1}\quad\text{and}\quad
 \sup_{t\in I}\| \Phi(u) \|_{H_x^1} \leq C\|u_0\|_{H^1} + CM^{\beta+1}.$$
Hence we get $\Phi(u)\in X$ for $u\in X$ if
\begin{equation}\label{w}
\varepsilon + C M^{\beta+1}\leq M\quad \textnormal{and} \quad C\|u_0\|_{H^1} + C M^{\beta+1} \leq N.
\end{equation}

Next we show that $\Phi$ is a contraction on $X$.
Using the same arguments used in \eqref{zx} and \eqref{zx2}, we see
\begin{align*}
\|\Phi(u)-\Phi(v)\|_{S_\gamma(I)} \leq C\|F(u)-F(v)\|_{S'_{\tilde{\gamma}}(I)}
\end{align*}
and
\begin{align*}
 \sup_{t\in I}\|\Phi(u)-\Phi(v)\|_{L_x^2} \leq C\|F(u)-F(v)\|_{S'_{\tilde{\gamma}}(I)}.
\end{align*}
By applying Lemma \ref{le} with the simple inequality
$\big|\,|u|^{\beta}u-|v|^{\beta}v\big|\lesssim \big(|u|^{\beta}+|v|^{\beta}\big)|u-v|$, we see
\begin{align*}
\|F(u)-F(v)\|_{S'_{\tilde{\gamma}}(I)}
&\leq C \big(\| \nabla u\|_{S_{\gamma}(I)}^{\beta}+ \| \nabla v\|_{S_{\gamma}(I)}^{\beta}\big) \|u-v\|_{S_{\gamma}(I)}\\
&\leq CM^{\beta}\|u-v\|_{S_{\gamma}(I)}
\end{align*}
as in \eqref{asb}.
Hence, for $u,v\in X$ we obtain $d(\Phi(u), \Phi(v)) \leq CM^{\beta} d(u, v)$.
Now by taking $N= 2C\|u_0\|_{H^1}$ and $M=2 \varepsilon$ and then choosing $\varepsilon>0$ small enough so that
\eqref{w} holds and $CM^\beta \leq1/2$,
it follows that $X$ is stable by $\Phi$ and $\Phi$ is a contraction on $X$.
Therefore, we have proved that there exists a unique local solution with
$u \in C(I ; H^1) \cap L^{q}(I ; H^{1,r}(|x|^{-r \gamma}))$
for any $(q, r) \in A_{\gamma}$.

\subsubsection*{Proof of Theorem \ref{thm2}}
Using \eqref{p1}, we observe that \eqref{ep} is satisfied also if $\|u_0\|_{H^1}$ is sufficiently small;
$$\|e^{it\Delta}u_0\|_{\mathcal{H}_{\gamma}(I)} \leq C\|u_0\|_{H^1} \leq \varepsilon$$
from which one can take $T=\infty$ in the above argument to obtain a global unique solution.
To prove the scattering property, we first note that
\begin{align*}
\big\|e^{-it_2\Delta}u(t_2)-e^{-it_1\Delta}u(t_1)\big\|_{H_x^1}&=\bigg\|\int_{t_1}^{t_2}e^{-i\tau\Delta}F(u)d\tau\bigg\|_{H_x^1}\\
&\lesssim\|F(u)\|_{\mathcal{H}'_{\tilde{\gamma}}([t_1, t_2])}\\
&\lesssim\| u \|_{\mathcal{H}_{\gamma}([t_1, t_2])} ^{\beta+1}\quad\rightarrow\quad0
\end{align*}
as $t_1,t_2\rightarrow\infty$.
This implies that
$\varphi:=\lim_{t\rightarrow\infty}e^{-it\Delta}u(t)$ exists in $H^1$. Furthermore,
$$u(t)-e^{it\Delta}\varphi= i\lambda  \int_{t}^{\infty} e^{i(t-\tau)\Delta}F(u) d\tau,$$
and hence
\begin{align*}
\big\|u(t)-e^{it\Delta}\varphi\big\|_{H_x^1}
&=\bigg\|\int_{t}^{\infty} e^{i(t-\tau)\Delta}F(u) d\tau\bigg\|_{H_x^1}\\
&\lesssim\|F(u)\|_{\mathcal{H}'_{\tilde{\gamma}}([t, \infty) )}\\
&\lesssim\| u \|_{\mathcal{H}_{\gamma}([t, \infty) )} ^{\beta+1} \quad\rightarrow\quad0
\end{align*}
as $t\rightarrow\infty$.
This completes the proof.

\subsubsection*{Note added in proof}
The continuous dependence of the solution $u$ with respect to the initial data $u_0$ follows clearly in the same way:
\begin{align*}
d(u,v)&\lesssim d\big(e^{it \Delta}u_0, e^{it\Delta} v_0\big) +d\bigg(\int_{0}^{t} e^{i(t-\tau)\Delta}F(u)d\tau, \int_{0}^{t} e^{i(t-\tau)\Delta}F(v) d\tau\bigg)\\
&\lesssim\|u_0-v_0\|_{L^2} + \frac{1}{2} d(u,v)
\end{align*}
which implies $$d(u,v) \lesssim   \|u_0-v_0\|_{H^1}.$$
Here, $u,v$ are the corresponding solutions for initial data $u_0,v_0$, respectively.

\section{Appendix}
In this final section we show that $(X,d)$ is a complete metric space.

Let $\{u_k\}$ be a Cauchy sequence in $(X, d)$.
Then it also becomes a Cauchy sequence in $C_t(I; L_x^2) \cap L^q_t(I; L_x^r(|x|^{-r\gamma}))$.
Since this space is complete\footnote{Since $|x|^{-r\gamma}$ is a locally integrable and nonnegative function in $\mathbb{R}^n$ if $\gamma<n/r$, we can define a Radon measure $\mu$ which is canonically associated with $|x|^{-r\gamma}$ by
$\mu(E)=\int_{E} |x|^{-r\gamma} dx$, $E \subseteq \mathbb{R}^n$,
so that $d\mu(x)=|x|^{-r\gamma} dx$. (See p. 5, \cite{HKM} for details.)
Hence we may regard $L^r(|x|^{-r\gamma})$ as $L^r(d\mu)$.},
there exists $u\in C_t(I; L_x^2) \cap L^q_t(I; L_x^r(|x|^{-r\gamma}))$ such that
$d(u_k,u)\rightarrow0$ as $k \rightarrow \infty$.
Namely,
\begin{equation}\label{dis}
\sup_{t \in I}\|u_k-u\|_{ L_x^2} +\|u_k-u\|_{L_t^{q}(I;L_x^r(|x|^{-r\gamma}))} \rightarrow 0
\end{equation}
for all $\gamma$-\textit{Schr\"odinger admissible} $(q, r)$.

Now it is enough to show $u\in X$.
For this we shall use the fact that
every bounded sequence in a reflexive space has a weakly convergent subsequence.
Since $u_k \in C_t(I;H_x^1)$,
we see that $u_k(t) \in H_x^1$ for almost all $t \in I$
and $\|u_k(t)\|_{H_x^1}\leq N$.
Since $H_x^1$ is reflexive, there exists a subsequence, which we still denote by $u_k(t)$,
such that $u_k(t) \rightharpoonup v(t)$ in $H_x^1$
and
\begin{equation*}
\| v(t)\|_{ H_x^1 } \leq \liminf_{k\rightarrow \infty} \|u_k(t)\|_{H_x^1}\leq N.
\end{equation*}
On the other hand, by \eqref{dis}, $u_k(t) \rightarrow u(t)$ in $ L_x^2$.
By uniqueness of limit, we conclude $u(t)=v(t)$ and therefore
$\|u(t)\|_{H_x^1} \leq N$.
Consequently, $u \in C_t(I; H_x^1)$ with $\sup_{t\in I} \|u\|_{H_x^{1}}\leq N$.

It remains to show that $u \in L_t^q (I; H^{1,r}_x(|x|^{-r\gamma}))$ with $\|u\|_{\mathcal{H}_{\gamma}(I)} \leq M$.
For this we shall apply the following lemma with $Y=H_x^{1,r} (|x|^{-r\gamma})$ and $Z=L_x^r(|x|^{-r\gamma})$.
\begin{lem}(Theorem 1.2.5, \cite{C})\label{lem4}
Consider two Banach spaces $Y \hookrightarrow Z$ and $1<q \leq \infty$.
Let $(f_k)_{k\ge0}$ be a bounded sequence in $L^q(I;Z)$ and
let $f:I \rightarrow Z$ be such that $f_k(t)\rightharpoonup f(t)$ in $Z$ as $k \rightarrow \infty$, for almost all $t \in I$.
If $(f_k)_{k\ge0}$ is bounded in $L^q(I; Y)$ and
if $Y$ is reflexive, then $f \in L^q(I; Y)$ and $\|f\|_{L^q(I; Y)} \leq \liminf_{k\rightarrow \infty} \|f_k\|_{L^q(I; Y)}.$
\end{lem}

Indeed, since $u_k \in X$ is a bounded sequence in $L_t^{q}(I;L_x^r(|x|^{-r\gamma}))$, we first note that
\begin{equation*}
\|u\|_{L_t^{q}(I;L_x^r(|x|^{-r\gamma}))} \leq \|u-u_k\|_{L_t^{q}(I;L_x^r(|x|^{-r\gamma}))}+\|u_k\|_{L_t^{q}(I;L_x^r(|x|^{-r\gamma}))}\leq M
\end{equation*}
as $k\rightarrow\infty$, so that $u \in L_t^q(I; L^r_x(|x|^{-r\gamma}))$.
Thus we see that  $u(t) \in L_x^r( |x|^{-r\gamma})$ for almost all $t \in I$ and  $u_k(t) \rightarrow u(t)$ in $L_x^r( |x|^{-r\gamma})$.
On the other hand, we see that $u_k$ is bounded in $L_t^q(I; H_x^{1,r}(|x|^{-r\gamma}))$ as
$$\| u_k\|_{L_t^q(I; H_x^{1,r}(|x|^{-r\gamma}))} \leq \|u_k\|_{\mathcal{H}_{\gamma}}\leq M.$$
Then, since $H_x^{1, r}(|x|^{-r\gamma})$ is reflexive as long as $1<r<\infty$ and $r\gamma<n$ (see p. 13, \cite{HKM}),
the above lemma implies now
\begin{equation*}
u \in L_t^q(I;H_x^{1,r} (|x|^{-r\gamma}))
\end{equation*}
and
$$\| u\|_{L_t^q (I; H_x^{1,r} (|x|^{-r\gamma}))} \leq \liminf_{k\rightarrow \infty} \|u_k\|_{L_t^q (I; H^{1,r}_x (|x|^{-r\gamma})) } \leq M$$
for all $\gamma$-\textit{Schr\"odinger admissible} $(q, r)$.
Hence, we get $\|u\|_{\mathcal{H}_{\gamma}(I)} \leq M$.

\

\noindent\textbf{Acknowledgment.}
The authors would like to thank Yonggeun Cho for informing us of the paper \cite{CHL} where
the energy-critical case in three dimensions was treated under radial symmetry.


\begin{thebibliography}{9}

\bibitem{BK} M. Ben-Artzi and S. Klainerman, \textit{Decay and regularity for the Schr\"odinger equation},
J. Anal. Math. 58 (1992), 25-37.

	\bibitem{B} L. Berg\'e, \textit{Soliton stability versus collapse}, Phys. Rev. E, 62 (2000), 3071-3074.


	\bibitem{BL} J. Bergh and J. L\"ofstr\"om, \textit{Interpolation Spaces, An Introduction}, Springer-Verlag, Berlin-New York, 1976.

	   \bibitem{CKN} L. Caffarelli, R. Kohn and L. Nirenberg, \textit{First order interpolation inequalities with weights}, Compositio Math. 53 (1984) 259-275.

	\bibitem{C} T. Cazenave, \textit{Semilinear Schr\"odinger equations}, Courant Lecture Notes in Mathematics, 10. New York University, Courant Institute of Mathematical Sciences, New York; American Mathematical Society, Providence, RI, 2003.

\bibitem{CW} T. Cazenave and F. B. Weissler, \textit{Some remarks on the nonlinear Schr\"odinger equation in the critical case}, Nonlinear Semigroups, Partial differential equations and attractors, Lecture Notes in Math. 1394, Springer, Berlin, 1989, 18-29.


\bibitem{CHL} Y. Cho, S. Hong and K. Lee, 
\textit{On the global well-posedness of focusing energy-critical inhomogeneous NLS}, 
J. Evol. Equ. 20 (2020), 1349-1380.

	\bibitem{CK} M. Christ and A. Kiselev, \textit{Maximal functions associated to filtrations}, J. Funct. Anal. 179 (2001), 409-425.

	\bibitem{D} V. D. Dinh, \textit{Scattering theory in a weighted $L^2$ space for a class of the defocusing inhomogeneous nonlinear Schr\"odinger equation}, Preprint, arXiv:1710.01392v1.

	\bibitem{F} L. G. Farah, \textit{Global well-posedness and blow-up on the energy space for the inhomogeneous nonlinear Schr\"odinger equation}, J. Evol. Equ. 16 (2016), 193-208.

	\bibitem{GS} F. Genoud and C. A. Stuart, \textit{Schr\"odinger equations with a spatially decaying nonlinearity: existence and stability of standing waves}, Discrete Contin. Dyn. Syst. 21 (2008), 137-186.

	\bibitem{GV3} J. Ginibre and G. Velo, \textit{On a class of nonlinear Schr\"{o}dinger equations I. The Cauchy problem, general case}, J. Funct. Anal. 32 (1979), 1-32.

	\bibitem{GV4}  J. Ginibre and G. Velo, \textit{The global Cauchy problem for the nonlinear Schr\"{o}dinger equation revisited}, Ann. Inst. H. Poincar\'{e} Anal. Non Lin\'{e}aire, 2 (1985), 309-327.

	\bibitem{GU} C. M. Guzm\'{a}n, \textit{On well posedness for the inhomogeneous nonlinear Schr\"{o}dinger equation}, Nonlinear Anal. Real World Appl. 37 (2017), 249-286.

\bibitem{HKM} J. Heinonen, T. Kilpe\"{a}inen and O. Martio, \textit{Nonlinear potential theory of degenerate elliptic equations}, Oxford Mathematical Monographs, Oxford University Press, 1993.

\bibitem{HR} J. Holmer and S. Roudenko, \textit{A sharp condition for scattering of the radial 3D cubic nonlinear Schr\"{o}dinger equation}, Comm. Math. Phys. 282 (2008), 435-467.


\bibitem{K} T. Kato, \textit{On nonlinear Schr\"{o}dinger equations}, Ann. Inst. H. Poincar\'{e} Phys. Th\'{e}or. 46 (1987), 113-129.

 \bibitem{KY} T. Kato and K. Yajima, \textit{Some examples of smooth operators and the associated smoothing effect}, Rev. Math. Phys. 1 (1989), 481-496.	

\bibitem{KT} M. Keel and T. Tao, \textit{Endpoint Strichartz estimates}, Amer. J. Math. 120 (1998), 955-980.

\bibitem{KLS} J. Kim, Y. Lee and I. Seo, \textit{On well-posedness for the inhomogeneous nonlinear Schr\"{o}dinger equation in the critical case}, 
J. Differential Equations 280 (2021), 179-202.

\bibitem{MS} V. G. Ma\'zya and T. O. Shaposhnikova, \textit{Theory of Sobolev multipliers. With applications to differential and integral operators}, Grundlehren der Mathematischen Wissenschaften (Fundamental Principles of Mathematical Sciences), 337. Springer-Verlag, Berlin, 2009.

	\bibitem{St} R. S. Strichartz, \textit{Restrictions of Fourier transforms to quadratic surfaces and decay of solutions of wave equations}, Duke Math. J. 44 (1977), 705-714.

    \bibitem{Su} M. Sugimoto, \textit{Global smoothing properties of generalized Schr\"odinger equations},
J. Anal. Math. 76 (1998), 191-204.

	\bibitem{TM} I. Towers and B. A. Malomed, \textit{Stable (2+1)-dimensional solitons in a layered medium with sign-alternating Kerr nonlinearity}, J. Opt. Soc. Amer. B Opt. Phys. 19 (2002), 537-543.

	\bibitem{V} M. C. Vilela, \textit{Regularity of solutions to the free Schr\"odinger equation with radial initial data}, Illinois J. Math. 45 (2001), 361-370.
	
\bibitem{W} K. Watanabe, \textit{Smooth perturbations of the selfadjoint operator $|\Delta|^{\alpha/2}$}, Tokyo J. Math. 14 (1991), 239-250.

\end{thebibliography}
\end{document}